\documentclass[12pt, a4paper, abstracton, bibliography=totoc]{scrartcl}
\pdfoutput=1

\usepackage{amsmath, amsthm, amsfonts, amssymb}
\usepackage[utf8]{inputenc}
\usepackage{slashed} 
\usepackage[all, cmtip]{xy}
\usepackage{hyperref} 

\let\counterwithout\relax

\usepackage{chngcntr} 
\usepackage{graphicx} 
\usepackage[english]{babel}
\usepackage{microtype}

\newcommand{\IN}{\mathbb{N}}
\newcommand{\IZ}{\mathbb{Z}}
\newcommand{\IQ}{\mathbb{Q}}
\newcommand{\IR}{\mathbb{R}}
\newcommand{\IC}{\mathbb{C}}

\newcommand{\id}{\operatorname{id}}

\newcommand{\card}{\#}

\newcommand{\ch}{\operatorname{ch}}
\newcommand{\vol}{\operatorname{vol}}
\newcommand{\ind}{\operatorname{ind}}

\newcommand{\RHufG}{RH^\mathrm{uf,\Gamma}}

\newcommand{\RHulfG}{RH^\mathrm{ulf,\Gamma}}

\newcommand{\Huf}{H^\mathrm{uf}}
\newcommand{\Hulf}{H^\mathrm{ulf}}
\newcommand{\Hufpol}{H^\mathrm{pol}}

\newcommand{\HufG}{H^\mathrm{uf,\Gamma}}
\newcommand{\HufL}{H^\mathrm{uf,\Lambda}}
\newcommand{\HulfG}{H^\mathrm{ulf,\Gamma}}
\newcommand{\HulfL}{H^\mathrm{ulf,\Lambda}}
\newcommand{\Cuf}{C^\mathrm{uf}}
\newcommand{\Culf}{C^\mathrm{ulf}}
\newcommand{\Cufpol}{{C}^\mathrm{pol}}

\newcommand{\uf}{\mathrm{uf}}

\newcommand{\HbdR}{H_{b, \mathrm{dR}}}

\newcommand{\length}{\mathrm{length}}

\newcommand{\PD}{\operatorname{PD}}

\newcommand{\Frechet}{Fr\'{e}chet }
\newcommand{\Poincare}{Poincar\'{e} }

\newcommand{\alg}{\mathrm{alg}}
\newcommand{\Kalg}{K^\alg}

\def\colim{\mathop{\mathrm{colim}}\nolimits}

\newcommand*{\largecdot}{\raisebox{-0.25ex}{\scalebox{1.2}{$\cdot$}}}

\newtheorem{thm}{Theorem}[section]
\newtheorem*{thm*}{Theorem}
\newtheorem*{mainthm*}{Main Theorem}

\newtheorem*{roesthm*}{Roe's Index Theorem}

\newtheorem*{atiyahsingerthm*}{Atiyah--Singer Index Theorem}

\newtheorem{cor}[thm]{Corollary}
\newtheorem*{cor*}{Corollary}
\newtheorem*{cor310*}{Corollary~\ref{cor33223}}
\newtheorem{lem}[thm]{Lemma}
\newtheorem{prop}[thm]{Proposition}

\newtheorem{question}[thm]{Question}
\theoremstyle{definition}
\newtheorem{defn-alt}[thm]{Definition}
\newtheorem{defns-alt}[thm]{Definitions}
\newtheorem{nota-alt}[thm]{Notation}
\newtheorem{rem-alt}[thm]{Remark}
\newtheorem{example-alt}[thm]{Example}
\newtheorem{examples-alt}[thm]{Examples}
\newtheorem{nonexample-alt}[thm]{Non-Example}

\newenvironment{defn} 
{%
	\pushQED{\qed}\begin{defn-alt}}
	{\popQED\end{defn-alt}}

\newenvironment{defns} 
{%
	\pushQED{\qed}\begin{defns-alt}}
	{\popQED\end{defns-alt}}

\newenvironment{example} 
{%
	\pushQED{\qed}\begin{example-alt}}
	{\popQED\end{example-alt}}

\newenvironment{examples} 
{%
	\pushQED{\qed}\begin{examples-alt}}
	{\popQED\end{examples-alt}}

\newenvironment{nonexample} 
{%
	\pushQED{\qed}\begin{nonexample-alt}}
	{\popQED\end{nonexample-alt}}

\newenvironment{rem} 
{%
	\pushQED{\qed}\begin{rem-alt}}
	{\popQED\end{rem-alt}}
	
{%
	\pushQED{\qed}\begin{nota-alt}}
	{\popQED\end{nota-alt}}

\numberwithin{equation}{section}

\counterwithout{footnote}{section}

\makeatletter
\def\blfootnote{\gdef\@thefnmark{}\@footnotetext}
\makeatother

\begin{document}

\title{Wrong way maps in uniformly finite homology and homology of groups}
\author{Alexander Engel}
\date{}
\maketitle

\vspace*{-3.5\baselineskip}
\begin{center}
\footnotesize{
\textit{
Fakultät für Mathematik\\
Universität Regensburg\\
93040 Regensburg, Germany\\
\href{mailto:alexander.engel@mathematik.uni-regensburg.de}{alexander.engel@mathematik.uni-regensburg.de}
}}
\end{center}

\begin{abstract}
Given a non-compact Riemannian manifold $M$ and a submanifold $N \hookrightarrow M$ of codimension $q$, we will construct under certain assumptions on both $M$ and $N$ a wrong way map $\Huf_\ast(M) \to \Huf_{\ast-q}(N)$ in uniformly finite homology.

Using an equivariant version of the construction and applying it to universal covers, we will construct a map $H_\ast(B\pi_1 M) \to H_{\ast-q}(B\pi_1 N)$.

As applications we discuss obstructions to the existence of positive scalar curvature metrics and to inessentialness.
\blfootnote{\textit{$2010$ Mathematics Subject Classification.} Primary:\ 55N35; Secondary:\ 53C23,\ 58J22.}
\blfootnote{\textit{Keywords and phrases.} Coarse geometry, uniformly finite homology, homology of groups, essentialness, positive scalar curvature, wrong way maps.}
\end{abstract}

\tableofcontents

\section{Introduction}

Let us start by stating the topological version of the main result of the present paper. We will not state the rough\footnote{``Rough'' means ``uniformly coarse'' in this context. The term ``rough geometry'' was coined by Roe \cite[Exercise 1.12]{roe_lectures_coarse_geometry}.} version of the main result (Theorem~\ref{thm:constr_loc_map}) in the introduction since it requires several definitions which will be given in Section~\ref{secjk112323}.

In the following, we consider homology groups for some fixed coefficient group and correspondingly orientability is meant with respect to these coefficients.

\begin{mainthm*}[topological version, Theorem~\ref{thm:loc_BG}]
Let $M$ be a closed and connected manifold and let $N \hookrightarrow M$ be a closed, connected submanifold of codimension $q \ge 1$ such that the inclusion induces an injective map on fundamental groups $\pi_1(N) \hookrightarrow \pi_1(M)$.

Assume further one of the following:
\begin{enumerate}
\item That the normal bundle $\nu$ of the embedding $N \hookrightarrow M$ is oriented (for some fixed coefficient group) and, in the case $q \ge 2$, that $\pi_i(M) = 0$ for $2 \le i \le q$.
\item That the normal bundle $\nu$ of the embedding $N \hookrightarrow M$ is trivial.

Further, in the case $q = 2$ that $\pi_2(N) \to \pi_2(M)$ is surjective, and in the case $q > 2$ that $\pi_i(M) = 0$ for $2 \le i \le q-1$ and that $\pi_q(N) \to \pi_q(M)$ is surjective.
\end{enumerate}

Then we construct a homomorphism $H_\ast(B\pi_1 M) \to H_{\ast-q}(B\pi_1 N)$ such that the diagram
\[\xymatrix{
H_\ast(M) \ar[r] \ar[d] & H_\ast(D\nu, S\nu) \ar[r] & H_{\ast-q}(N) \ar[d]\\
H_\ast(B\pi_1 M) \ar[rr] & & H_{\ast-q}(B\pi_1 N)}\]
commutes, where the map $H_\ast(M) \to H_\ast(D\nu, S\nu)$ is the Thom--Pontryagin collapse, and the map $H_\ast(D\nu, S\nu) \to H_{\ast-q}(N)$ is cap-product with the Thom class of $\nu$. The vertical maps are induced by the classifying maps.\footnote{Recall that $B \pi_1 M$ is the classifying space for principal $\pi_1 M$-bundles. The classifying map $M \to B \pi_1 M$ is the one which classifies the universal cover of $M$; it is unique up to homotopy.}
\end{mainthm*}

One can show, using surgery theory, that for any finitely presented group $G$ and any finitely presented subgroup $H$ of it, the geometric situation of the main theorem can be achieved in codimension 2, providing a wrong way map $H_*(BG) \to H_{*-2}(BH)$ \cite{nitsche}.

\paragraph{Application A} As a first example of an application we will treat index theory.

Recall that if $M$ is spin, then it is oriented for $KO$-theory, i.e., has a fundamental class $[M] \in KO_\ast(M)$. The higher $\hat A$-genera of $M$ are defined as the image of $[M]$ under the composition $KO_\ast(M) \xrightarrow{\ch} H_\ast(M) \to H_\ast(B\pi_1 M)$, where we use $\IQ$-coefficients for the homology groups. The strong Novikov conjecture asserts that if $M$ admits a positive scalar curvature metric, then the higher $\hat A$-genera of $M$ must vanish.

Under the second set of assumptions of the main theorem, if additionally $M$ and $N$ are spin, the wrong way map $H_\ast(B\pi_1 M) \to H_{\ast-q}(B\pi_1 N)$ maps the higher $\hat A$-genera of~$M$ to the ones of $N$ (for this we need triviality of the normal bundle). We conclude:
\begin{center}
\noindent
\parbox[c][1.5em][c]{0.95\textwidth}{
\emph{If the strong Novikov conjecture is true for $\pi_1(M)$, then the higher $\hat A$-genera of $N$ are obstructions to the existence of metrics of positive scalar curvature on $M$.}
}
\end{center}

\paragraph{Application B} As a second application let us treat essentialness of oriented manifolds. Recall that an oriented manifold $M$ is called essential if the image $\varphi_\ast[M] \in H_\ast(B \pi_1 M)$ of its fundamental class $[M] \in H_\ast(M)$ under its classifying map $\varphi\colon M \to B \pi_1 M$ is non-zero. An interesting property of essential Riemannian manifolds is, e.g., Gromov's systolic inequality \cite{gromov_filling}: $\operatorname{sys}(M)^n \le C_n \vol(M)$, where $n = \dim(M)$, $C_n$ is a universal constant depending only on the dimension, and $\operatorname{sys}(M)$ denotes the smallest length of a homotopically non-trivial loop in $M$.

Now assume that we are in the geometric situation of the main theorem and that both $M$ and $N$ are orientable. From the fact that the composition of the Thom isomorphism with the Thom--Pontryagin collapse maps $[M]$ to $[N]$ we immediately conclude: 
\begin{center}
\noindent
\parbox[c][0.5em][c]{0.95\textwidth}{
\emph{If $N$ is essential, then $M$ must be also essential.}
}
\end{center}

An application of this criterion for essentialness is the following (the idea for this arose out of a discussion with Bernhard Hanke):
\begin{center}
\noindent
\parbox[c][4em][c]{0.95\textwidth}{
\emph{Assume that we are in the geometric situation of the main theorem, $M$ is spin, $N$ is orientable and essential, and the strong Novikov conjecture holds for $\pi_1(M)$, then $M$ does not admit a metric of positive scalar curvature.}
}
\end{center}

The interesting point is that we do not have to assume that the submanifold is spin. But in similar results like the ones of Hanke--Pape--Schick \cite{hanke_pape_schick} or Zeidler \cite{zeidler_codimension}, which we will survey in Section~\ref{secnlm24ed}, this is the case.

\paragraph{Question C} The following arose out of a discussion with Micha\l{} Marcinkowski.

Essentialness is a notion of largeness, but we also have enlargeability, hypereuclideaness or macroscopical largeness (many of these notions were introduced by Gromov, and enlargeability by Gromov--Lawson \cite{GL_enlargeable}). In Application~B we saw that if $M$ is small with respect to the notion of essentialness, then the submanifold $N$ must be also small. This leads to the question whether the same is true for the other notions of largeness.

Let us try to argue that for, e.g., hypereuclideaness this is probably not true. Recall \cite[Page~19]{gromov_pmsh} that a Riemannian manifold $M$ is called hypereuclidean if there exists a proper Lipschitz map $f \colon M \to \IR^n$ of non-zero degree. A closed manifold is called hypereuclidean if its universal cover is hypereuclidean, where we equip the closed manifold with some Riemannian metric and its universal cover with the pull-back metric (being hypereuclidean is then independent of the metric we put on the closed manifold).

Given an orientable aspherical manifold $M$ of dimension $n$, we can find an embedded $S^1 \hookrightarrow M$ which generates $\IZ < \pi_1 M$. Then all assumptions of our above main theorem are satisfied. Now $S^1$ is hypereuclidean and so we would have our counter-example finished if we would know that there exist aspherical manifolds which are not hypereuclidean. Unfortunately, as far as the author knows this is an open problem. The best guess is that an aspherical manifold containing coarsely an expander in its fundamental group should not be hypereuclidean (one can use Sapir's construction \cite{sapir} to get such a manifold), but this is currently unproven. Let us therefore phrase this as an open question:
\begin{center}
\noindent
\parbox[c][0.5em][c]{0.95\textwidth}{
\emph{Do there exist aspherical manifolds which are not hypereuclidean?}
}
\end{center}

\subsection{Related results by other authors}\label{secnlm24ed}

Let us summarize some results about constructing wrong way maps in homology theories associated to groups. The story started with the following result of Hanke--Pape--Schick that raised the author's interest. Note that a precursor of their result were Gromov's and Lawson's results in \cite[Section 7]{gromov_lawson_psc}.

\begin{thm*}[{\cite{hanke_pape_schick}}]
Let $M$ be a closed connected spin manifold with $\pi_2(M) = 0$.

Assume that $N \subset M$ is a codimension two submanifold with trivial normal bundle and that the induced map $\pi_1(N) \to \pi_1(M)$ is injective. Assume also that the Rosenberg index of $N$ does not vanish: $0 \not= \alpha(N) \in K_\ast(C^\ast \pi_1 (N))$.

Then $M$ does not admit a Riemannian metric of positive scalar curvature.
\end{thm*}

In \cite[Remark 1.9]{hanke_pape_schick} it was stated that if one assumes the strong Novikov conjecture for $\pi_1(M)$, then one can conclude indirectly (using the stable Gromov-Lawson-Rosenberg conjecture) that $\alpha(M)$ must be also non-zero. The technique employed in the proof of the above theorem could not show this directly. Indeed, it would be desirable to construct under the above assumptions a map $K_{\ast+2}(C^\ast \pi_1(M)) \to K_\ast(C^\ast \pi_1(N))$ which maps $\alpha(M)$ to $\alpha(N)$. This was recently achieved by Kubota \cite{kubota}.

Zeidler could solve the problem of constructing such a map in the following cases:

\begin{thm*}[{\cite[Theorems 1.5 \& 1.7]{zeidler_codimension}}]
A map $K_{\ast+q}(C^\ast \pi_1 (M)) \to K_\ast(C^\ast \pi_1(N))$ with $\alpha(M) \mapsto \alpha(N)$ exists in the following case, where $q$ is the codimension of $N \to M$:
\begin{itemize}
\item $N \to M \to B$ is a fiber bundle, $B$ is aspherical and $\pi_1(B)$ has finite asymptotic dimension.
\item $N$ has codimension $q=1$ and $\pi_1(N) \to \pi_1(M)$ is injective.
\end{itemize}
\end{thm*}

Working rationally, i.e., with the $\hat A$-genus instead of the $\alpha$-invariant, the author proved in an earlier paper the following result:

\begin{thm*}[{\cite{engel_rough}}]
Let $M$ be a closed, connected manifold with $\pi_1(M)$ virtually nilpotent and $\pi_i(M) = 0$ for $2 \le i \le q$. Assume furthermore that $N \subset M$ is a closed, connected submanifold of codimension $q$ and with trivial normal bundle.

Then $\hat A(N) \not= 0$ implies $\ind(\slashed D_X) \not= 0 \in K_\ast(C_u^\ast(X))$ for $X$ the universal cover of~$M$.
\end{thm*}

The proof of the previous theorem proceeds via rough index pairings, i.e., non-vanishing of $\ind(\slashed D_X)$ is witnessed by pairing with a coarse cohomology class related to $N$.

It was later noticed that in the above case the $\hat A$-genus of $N$ can be written as a higher $\hat A$-genus of $M$, which explains the result. But this idea was used by Zeidler to improve the previous theorem of the author to the following version:

\begin{thm*}[{\cite[Section 3]{zeidler_codimension}}]
Let $M$ be a closed, connected spin manifold and $N \subset M$ a closed, connected submanifold of codimension $q$ with trivial normal bundle. Assume also that $\pi_i(M) = 0$ for $2 \le i \le q$.

Then higher $\hat A$-genera of $N$ arising by pulling back elements from $H^\ast(B \pi_1(M))$ via the composition $N \to M \to B \pi_1 (M)$ are higher $\hat A$-genera of $M$.
\end{thm*}

In the above theorem one would like to have the same conclusion for all higher $\hat A$-genera of $N$, i.e., arising by pulling back elements from $H^\ast(B \pi_1 (N))$ via the classifying map of $N$. Note that the composition $N \to M \to B \pi_1 (M)$ is the same as $N \to B \pi_1 (N) \to B \pi_1(M)$, where the last map is induced by the map $\pi_1(N) \to \pi_1(M)$ induced from the inclusion $N \to M$. So if $H^\ast(B \pi_1 (M)) \to H^\ast(B \pi_1 (N))$ is not rationally surjective, the above result of Zeidler may miss some higher $\hat A$-genera of $N$. This is now basically what the topological version of the main theorem of the present paper corrects.

\paragraph{Acknowledgements}

I would like to thank Thomas Schick, Clara Löh, Ulrich Bunke and Bernhard Hanke for helpful discussions surrounding the ideas of this paper, and Rudolf Zeidler for pointing out an error in the previous versions. I also thank the anonymous referee for his or her comments.

I acknowledge the supported by the SFB 1085 ``Higher Invariants'' and the Research Fellowship EN 1163/1-1 ``Mapping Analysis to Homology'', both funded by the Deutsche Forschungsgemeinschaft DFG.

\section{Wrong way maps in uniformly finite homology}
\label{sec:2}

In this section we will discuss and prove the rough version of the topological main theorem from the introduction. The latter will be discussed in Section~\ref{subsecn2sd2} and follows immediately from an equivariant version of the rough theorem proved in Section~\ref{seck09723}.

\subsection{Rough version of the main theorem}\label{secjk112323}

Block and Weinberger \cite[Section 2]{block_weinberger_1} introduced uniformly finite homology whose definition we will recall in Definition~\ref{defnj23ds} below. The rough version of our main theorem will need a corresponding topological analogue of uniformly finite homology, which we are going to define first:

\begin{defn}[Uniformly locally finite homology]
Let $X$ be a metric space and $A$ a normed abelian group\footnote{That is to say, $A$ is equipped with a function $|\largecdot|\colon A \to \IR_{\ge 0}$ satisfying the triangle inequality.}. A uniformly locally finite $n$-chain with values in $A$ is a (possibly infinite) formal sum $\sum_{\alpha \in I} a_\alpha \sigma_\alpha$ with $a_\alpha \in A$ and $\sigma_\alpha \colon \Delta^n \to X$ continuous for all $\alpha \in I$, where $I$ is an index set, satisfying the following three conditions:
\begin{itemize}
\item $\sup_{\alpha \in I} |a_\alpha| < \infty$,
\item for every $r > 0$ there exists $K_r < \infty$ such that the ball $B_r(x)$ of radius $r$ around any point $x \in X$ meets at most $K_r$ simplices $\sigma_\alpha$, and
\item the family of maps $\{\sigma_\alpha\}_{\alpha \in I}$ is equicontinuous.
\end{itemize}
We equip the chain groups with the usual boundary operator and denote the resulting homology by $\Hulf_\ast(X)$, i.e., we will usually not mention the choice of $A$.
\end{defn}

\begin{rem}
Let $X$ be a simplicial complex of bounded geometry, i.e., the number of simplices in the link of each vertex is uniformly bounded, and equip $X$ with the metric derived from barycentric coordinates such that the edges all have length~$1$. Then we may define $L^\infty$-homology $H_\ast^\infty(X)$ by considering as chains (possibly infinite) formal sums of simplices with uniformly bounded coefficients (bounded geometry is needed so that the boundary operator is well-defined if the norm-function of $A$ is not bounded). Then we have $H^\infty_\ast(X) \cong \Hulf_\ast(X)$ and the map inducing the isomorphism is given by mapping a simplex of $X$ ``to itself'' but now viewed as a function $\Delta^n \to X$.
\end{rem}

\begin{defn}[Uniformly finite homology {\cite[Section 2]{block_weinberger_1}}]\label{defnj23ds}
Let $X$ be a metric space and $A$ a normed abelian group. A uniformly finite $n$-chain with values in $A$ is a (possibly infinite) formal sum $\sum_{\bar x \in X^{n+1}} a_{\bar x} {\bar x}$ with $a_{\bar x} \in A$ satisfying the following conditions:
\begin{itemize}
\item $\sup_{\bar x \in X^{n+1}} |a_{\bar x}| < \infty$,
\item for every $r > 0$ there exists $K_r < \infty$ such that $\#\{\bar x \in B_r(\bar y)\colon a_{\bar x} \not= 0\} < K_r$ for all points $\bar y \in X^{n+1}$, and
\item there exists $R >0$ such that $a_{\bar x} = 0$ if $d(\bar x, \Delta) > R$, where we denote the diagonal in $X^{n+1}$ by $\Delta = \{(x, \ldots, x)\colon x \in X\} \subset X^{n+1}$.
\end{itemize}
Equipping the chain groups $C_\ast^\uf(X)$ with the usual boundary operator (regarding a point $\bar x \in X^{n+1}$ as the vertices of an $n$-simplex in $X$) we get the uniformly finite homology groups $\Huf_\ast(X)$; again not mentioning $A$.
\end{defn}

Uniformly finite homology is functorial for so-called rough maps:

\begin{defn}[Rough maps]\label{defnjsd1f32}
A map $f\colon X \to Y$ between two metric spaces $X$ and $Y$ is called rough if for all $R > 0$ there exists an $S > 0$ such that we have the following two estimates:
\[d(x,y) < R \Rightarrow d(f(x),f(y)) < S \quad \text{ and } \quad d(f(x),f(y)) < R \Rightarrow d(x,y) < S.\]
Note that $f$ does not need to be continuous.
\end{defn}

We have a natural map $\Hulf_\ast(X) \to \Huf_\ast(X)$ by mapping a simplex $\sigma\colon \Delta^n \to X$ to its $(n+1)$-tuple of vertices $(\sigma(0), \ldots, \sigma(n)) \in X^{n+1}$.

\begin{defn}[Equicontinuously $q$-connected spaces]
\label{defn:equi_q_conn}
A metric space $X$ is said to be equicontinuously $q$-connected if for every $i \le q$ we have the following: any equicontinuous collection of maps $\{S^i \to X\}$, where $S^i \subset \IR^{i+1}$ is the standard $i$-sphere of radius $1$, is equicontinuously contractible.
\end{defn}

\begin{example}\label{exjnksd23}
Let $M$ be a closed, connected, Riemannian manifold with $\pi_i(M) = 0$ for $2 \le i \le q$. Then the universal cover of $M$ will be equicontinuously $q$-connected if equipped with the pull-back metric.
\end{example}

\begin{question}\label{question:uniform_implies_equi}
Under which conditions does it follow that uniform $q$-connectedness\footnote{This notion is derived from the notion of uniform contractibility, which is more often used in literature. It means that every continuous map $S^i \to B_r(x)$ is nullhomotopic in $B_{s}(x)$, where $s$ only depends on $r$ but not on $x \in X$.} implies equicontinuous $q$-connectedness?

Concretely, does this hold for simplicial complexes equipped with the natural metric derived from barycentric coordinates?
\end{question}

\begin{prop}\label{prop:contr_iso}
Let $X$ be an equicontinuously $q$-connected metric space.

Then the map $\Hulf_\ast(X) \to \Huf_\ast(X)$ is an isomorphism for $\ast \le q$.
\end{prop}

\begin{proof}
For any point $\bar x \in X^{i+1}$ with $i \le q+1$ we can construct inductively a simplex $\Delta(\bar x)\colon \Delta^i \to X$ with vertices $\bar x$ by exploiting the contractibility of~$X$: we first connect any two points of~$\bar x$ by a path, then we fill any $S^1$ that we get in this way by a $2$-disk, then we fill any $S^2$ that we get by a $3$-disk, and so on. Doing this inductively ensures that $\partial \Delta(\bar x) = \Delta (\partial \bar x)$, i.e., we have compatibility with the boundary operator.

Since $X$ is equicontinuously $q$-connected we can control the diameter of $\Delta(\bar x)$ and its modulus of continuity by the diameter of $\bar x$. So applying this procedure to a uniformly finite chain we get a uniformly locally finite chain. This map $\Delta$ will be a chain map due to the compatibility of the contruction with the boundary operator, and therefore we get a map $\Delta_\ast \colon \Huf_\ast(X) \to \Hulf_\ast(X)$ for all $\ast \le q$.

By contruction $\Delta_\ast$ is a right inverse for the natural map $\iota_\ast \colon \Hulf_\ast(X) \to \Huf_\ast(X)$ since on the level of chain groups we already have $\iota \circ \Delta = \id$. The composition $\Delta \circ \iota$ is not the identity map, but it is chain homotopic to it via a uniformly locally finite homotopy (i.e., a homotopy admissible for uniformly locally finite homology). So $\Delta_\ast$ is also the left inverse to $\iota_\ast$.
\end{proof}

\begin{defn}[ulf-submanifolds]
\label{defn:roughly_geodesic}
Let $(M,g)$ be a connected Riemannian manifold and let $N \subset M$ be a connected submanifold endowed with a Riemannian metric $h$ (not necessarily the one induced from $M$).

Then $N$ is a ulf-submanifold if the identity map $\id\colon (N, d_h) \to (N, d_g|_{N\times N})$ and its inverse $\id\colon (N, d_g|_{N\times N}) \to (N, d_h)$ are uniformly continuous and rough (Definition~\ref{defnjsd1f32}).

Here $d_h\colon N \times N \to \IR$ is the metric on $N$ derived from its Riemannian metric $h$ and $d_g|_{N\times N}$ is the subspace metric on $N$ induced from the metric space $(M, d_g)$.
\end{defn}

\begin{examples}
Any connected, totally geodesic submanifold is ulf.

Let $N \subset M$ be a submanifold, where $M$ is a closed Riemannian manifold. Let $X$ be the universal cover of $M$ equipped with the pull-back metric. Choose a connected component $\bar N \subset X$ of the preimage of $N$ under the covering projection $X \to M$ and equip $\bar N$ with the induced Riemannian metric. Then~$\bar N$ will be ulf, but usually not totally geodesic. Note that $\bar N$ will be also an example for Definition \ref{defn:thick}, i.e., it will have a uniformly thick normal bundle.
\end{examples}

\begin{rem}\label{remj23}
If $N \subset M$ is a ulf-submanifold, then $\Hulf_\ast(N)$ will not depend on the choice of metric $d_h$ or $d_g|_{N\times N}$ on $N$. The same holds for $\Huf_\ast(N)$.
\end{rem}

\begin{defn}[Uniformly thick normal bundles]\label{defn:thick}
Let $M$ be a Riemannian manifold and $N \subset M$ a submanifold. Its normal bundle $\nu$ is uniformly thick if there is an $\varepsilon > 0$ such that the exponential map of $M$ is an embedding of $\{V \in \nu: \|V\| < \varepsilon\}$ into $M$.
\end{defn} 

\begin{nonexample}
There is no isometric embedding of hyperbolic space $\mathbb{H}^n$ into $\IR^N$ with a uniformly thick normal bundle. The reason is that if we had such an embedding, we could bound from above the volume growth of hyperbolic space by the volume growth of Euclidean space. I learned this on MathOverflow from Anton Petrunin. Note that without the requirement of having a uniformly thick normal bundle there is an isometric embedding $\mathbb{H}^n \to \IR^N$ by the Nash embedding theorem.
\end{nonexample}

\begin{lem}
Let $N$ be a ulf-submanifold of $M$ of codimension $q \ge 1$ with a uniformly thick, oriented\footnote{With respect to the chosen coefficient group $A$.} normal bundle $\nu$. Equip the disc bundle $D\nu$ with the induced metric.

Then the Thom map $\Hulf_\ast(D\nu, S\nu) \to \Hulf_{\ast-q}(N)$ is well-defined.
\end{lem}

\begin{proof}
One has to check that the cap-product of a ulf-chain with the Thom class is again a ulf-chain. This follows from the uniform thickness of the normal bundle: for example, given a ulf-chain $c$, there is a uniform upper bound on the number how often any simplex of $c$ can cross the whole disc bundle. The latter implies that the cap-product of $c$ with the Thom class has again uniformly bounded coefficients.

The Thom map maps into $\Hulf_{\ast-q}(N)$, where to define this we use the metric on $N$ coming from the surrounding metric of $M$ (since we use for the disc bundle $D\nu$ the from $M$ induced metric). We need that $N$ is a ulf-submanifold to be sure that changing the metric to the given one of $N$ induces an isomorphism on ulf-homology, cf.~Remark~\ref{remj23}.
\end{proof}

Let us state now the last definition we need before we are able to state and prove the main theorem in its rough version.

\begin{defn}\label{defnjknfwe2323}
Let $X \subset Y$ be a subspace of the metric space $Y$.

We say that the induced map $\pi_q(X) \to \pi_q(Y)$ is equicontinuously surjective, if for all $R > 0$ we have the following: given an equicontinuous collection of maps $\{S^q \to B_R(X)\}$, then there exist homotopies pushing these spheres into $X$ and such that these homotopies form an equicontinuous family, too.

Here $B_R(X) \subset Y$ denotes the $R$-neighbourhood of $X$ inside $Y$.
\end{defn}

\begin{thm}\label{thm:constr_loc_map}
Let $M$ be a Riemannian manifold, and let $N \subset M$ be a ulf-submanifold of codimension $q \ge 1$ with a uniformly thick normal bundle $\nu$.

Assume further one of the following:
\begin{enumerate}
\item That $M$ is equicontinuously $q$-connected and $\nu$ oriented\footnote{for the chosen coefficients $A$}.
\item That $M$ is equicontinuously $(q-1)$-connected, $\pi_q(N) \to \pi_q(M)$ is equicontinuously surjective and $\nu$ is trivial.
\end{enumerate}

Then there is a map $\Huf_\ast(M) \to \Huf_{\ast-q}(N)$ such that the following diagram commutes:
\[\xymatrix{
\Hulf_\ast(M) \ar[r] \ar[d] & \Hulf_\ast(D\nu, S\nu) \ar[r] & \Hulf_{\ast-q}(N) \ar[d]\\
\Huf_\ast(M) \ar[rr] & & \Huf_{\ast-q}(N)}\]
where the map $\Hulf_\ast(M) \to \Hulf_\ast(D\nu, S\nu)$ is the Thom--Pontryagin collapse\footnote{In order to define this collapse properly we need a corresponding excision result for ulf-homology. We have it here since we assume the normal bundle to be uniformly thick.}, and the map $\Hulf_\ast(D\nu, S\nu) \to \Hulf_{\ast-q}(N)$ is cap-product with the Thom class of $\nu$.
\end{thm}

\begin{proof}
For $k \ge q$ let $c \in C_k^\uf(M)$ with $c = \sum_{\bar x \in M^{k+1}} a_{\bar x} {\bar x}$.

For $\bar x = (x_0, \ldots, x_k)$ we form the $q$-simplex $\Delta(x_0, \ldots, x_q)$, where $\Delta$ is as in the proof of Proposition \ref{prop:contr_iso}. Here we need the assumption that $M$ is $(q-1)$-connected; otherwise we would not be able to construct this $q$-simplex. If $\theta$ denotes the Thom class of $\nu$, we may then integrate $\Delta(x_0, \ldots, x_q)$ against it to get $\theta(\Delta(x_0, \ldots, x_q)) \in A$. So for given $\bar x = (x_0, \ldots, x_k)$ we have constructed
\[\theta \cap \bar x := \theta(\Delta(x_0, \ldots, x_q)) \cdot (x_q, \ldots, x_k).\]

We set $\theta \cap c := \sum_{\bar x \in M^{k+1}} a_{\bar x} (\theta \cap \bar x)$. Because $c$ is a uniformly finite $k$-chain, $\theta$ a uniform class, and $M$ equicontinuously $(q-1)$-connected, we conclude that $\theta \cap c$ is a uniformly finite $(k-q)$-chain which is supported in an $R$-neighbourhood of $N$.

Let $\eta\colon M \to N$ map $x \in M$ to a point $y \in N$ which minimizes the distance from $x$ to $N$, i.e., $d_M(x,N) = d_M(x,y)$. Such a point $y$ might not be unique; in this case we just choose one. Now we set $\eta_\ast(x_q, \ldots, x_k) := (\eta(x_q), \ldots, \eta(x_k)) \in N^{k-q+1}$ and define the map $C_\ast^\uf(M) \to C_{\ast-q}^\uf(N)$ by $c \mapsto \eta_\ast(\theta \cap c)$. Note that since $\theta \cap c$ is a uniformly finite chain of $M$ supported in an $R$-neighbourhood of $N$, we can conclude that $\eta_\ast(\theta\cap c)$ is a uniformly finite chain on $N$, where we use on $N$ the from $M$ induced metric. Since $N$ is a ulf-submanifold, $\eta_\ast(\theta\cap c)$ will stay a uniformly finite chain if we change the metric on $N$ to its original metric, cf.~Remark~\ref{remj23}.

Let us show that the above constructed map descends to homology classes. In the case where we assume that $M$ is equicontinuously $q$-connected, we can carry out the construction also for $(q+1)$-simplices and hence our map $C_\ast^\uf(M) \to C_{\ast-q}^\uf(N)$ will be a chain map (because the construction of $\Delta$ is inductively over the skeleton of a simplex, so that it becomes compatible with the boundary operator).

The case where we assume that $M$ is equicontinuously $(q-1)$-connected, the map $\pi_q(N) \to \pi_q(M)$ is equicontinuously surjective and $\nu$ is trivial, is a bit more involved. Here we will directly show $\partial (\eta_\ast (\theta \cap c)) = (-1)^q \cdot \eta_\ast (\theta \cap \partial c)$. The formulas we get are
\begin{align*}
\partial (\eta_\ast (\theta \cap c)) & = \sum_{\bar x \in M^{k+1}} a_{\bar x} \cdot \theta(\Delta(x_0, \ldots, x_q)) \cdot \sum_{j = q}^k (-1)^{j-q} \cdot (\eta(x_q), \ldots, \widehat{\eta(x_j)}, \ldots, \eta(x_k)),\\
\eta_\ast (\theta \cap \partial c) & = \sum_{\bar x \in M^{k+1}} a_{\bar x} \cdot \sum_{j = 0}^k (-1)^j \cdot \eta_\ast (\theta \cap (x_0, \ldots, \widehat{x_j}, \ldots, x_k)),
\end{align*}
where the second one can be further expanded by using
\begin{align*}
\sum_{j = 0}^k & (-1)^j \cdot \eta_\ast (\theta \cap (x_0, \ldots, \widehat{x_j}, \ldots, x_k))\\
& = \sum_{j = 0}^q \, (-1)^j \theta(\Delta(x_0, \ldots, \widehat{x_j}, \ldots, x_{q+1})) \cdot (\eta(x_{q+1}), \ldots, \eta(x_k)) \ \! +\\
& \quad \! \sum_{j=q+1}^k \! (-1)^j \theta(\Delta(x_0, \ldots, x_q)) \cdot (\eta(x_q), \ldots, \widehat{\eta(x_j)}, \ldots, \eta(x_k).
\end{align*}
So in order to have $\partial (\eta_\ast (\theta \cap c)) = (-1)^q \cdot \eta_\ast (\theta \cap \partial c)$ we conclude that we need
\[\sum_{\bar x \in M^{k+1}} a_{\bar x} \cdot \sum_{j = 0}^{q+1} (-1)^j \theta(\Delta(x_0, \ldots, \widehat{x_j}, \ldots, x_{q+1})) \cdot (\eta(x_{q+1}), \ldots, \eta(x_k)) \stackrel{!}= 0.\]
Note that $\sum_{j = 0}^{q+1} (-1)^j \Delta(x_0, \ldots, \widehat{x_j}, \ldots, x_{q+1})$ is a cycle (it would be the boundary of the simplex $\Delta(x_0, \ldots, x_{q+1})$ if $M$ would be $q$-connected). We regard it as a $q$-sphere, and so
\begin{equation}\label{eqkjsfd23}
\sum_{\bar x \in M^{k+1}} a_{\bar x} \cdot \sum_{j = 0}^{q+1} (-1)^j \Delta(x_0, \ldots, \widehat{x_j}, \ldots, x_{q+1})
\end{equation}
can be seen as an equicontinuous family of $q$-spheres supported in an $R$-neighbourhood of $N$ (forget for a moment the coefficients $a_{\bar x}$). Because we assume $\pi_q(N) \to \pi_q(M)$ to be equicontinuously surjective, we can conclude that~\eqref{eqkjsfd23} is homologous to a degree~$q$ ulf-cycle supported in $N$. Hence the application of $\theta$ to it vanishes, since we assume the normal bundle to be trivial.

The above arguments give us the desired map $H_\ast^\uf(M) \to H_{\ast-q}^\uf(N)$. Commutativity of the main diagram is clear from its construction. Let us argue that it is independent of any choices made: to show that different filling maps $\Delta\colon \Cuf_q(M) \to \Culf_q(M)$ result in the same map $H_\ast^\uf(M) \to H_{\ast-q}^\uf(N)$ we have to use the same argument that we used to show that the map is well-defined on homology classes. Different maps $\eta|_{B_R(N)} \colon B_R(N) \to N$ for fixed $R > 0$ are close to each other and therefore the resulting uniformly finite homology class is independent of the choice of $\eta$. Note that the restriction of $\eta$ to the $R$-neighbourhood of $N$ is necessary for the statement to be true, but this restriction is no problem for the proof since $\theta \cap c$ is supported in an $R$-neighbourhood of $N$.
\end{proof}

\begin{rem}\label{rem:more_maps}
If for some $q^\prime \ge q$ one of the assumptions in Theorem~\ref{thm:constr_loc_map} is satisfied for $q^\prime$ instead of $q$, we may use any uniform\footnote{This just means that the cap-product is well-defined, i.e., maps ulf-classes to ulf-classes.} cohomology class of degree $q^\prime$ of the pair $(D\nu, S\nu)$ for the cap-product map $\Hulf_\ast(D\nu, S\nu) \to \Hulf_{\ast-q^\prime}(N)$ to get a corresponding map $\Huf_\ast(M) \to \Huf_{\ast-q^\prime}(N)$.
\end{rem}

\subsection{Continuity of the wrong way map}
\label{sec:3}

In this section we will discuss continuity of the rough wrong way map, where we will equip the uniformly finite homology groups with certain semi-norms. We will apply this continuity result in Section~\ref{seckjnsd}.

Let $c > 0$ and let $\Gamma_N \subset N$ be a maximal $c$-separated subset of $N$, i.e., $d(\gamma, \gamma^\prime) > c$ for any $\gamma \not= \gamma^\prime \in \Gamma_N$. The inclusion $\Gamma_N \to N$ induces an isomorphism $\Huf_\ast(\Gamma_N) \cong \Huf_\ast(N)$ where the inverse map is induced by the map $N \to \Gamma_N$ given by mapping a point of $N$ to the nearest point of $\Gamma_N$ (in case this nearest point is not unique just pick one arbitrarily). This is \cite[Corollary 2.2]{block_weinberger_1}.

Choosing discretizing subsets $\Gamma_M \subset M$ and $\Gamma_N \subset N$, we can construct under the assumptions of Theorem \ref{thm:constr_loc_map} a map $\Huf_\ast(\Gamma_M) \to \Huf_{\ast-q}(\Gamma_N)$ such that we get the diagram
\[\xymatrix{
\Huf_\ast(M) \ar[rr] \ar[d] & & \Huf_{\ast-q}(N) \ar[d]\\
\Huf_\ast(\Gamma_M) \ar[rr] & & \Huf_{\ast-q}(\Gamma_N)}\]
The construction of the map $\Huf_\ast(\Gamma_M) \to \Huf_{\ast-q}(\Gamma_N)$ is analogous to the construction of $\Huf_\ast(M) \to \Huf_{\ast-q}(N)$ with the only change that at the end we project by $\eta$ to $\Gamma_N$.

The reason for doing the discretization is that now we may use the ideas from \cite{engel_rough}.

\begin{defn}[{\cite{engel_rough}}]
\label{defn_pol_chains}
For every $n \in \IN_0$ we define the following norm of a uniformly finite chain $c = \sum a_{\bar y} {\bar y} \in \Cuf_k(Y)$ of a uniformly discrete space $Y$:
\[\|c\|_{\infty,n} := \sup_{\bar y \in Y^{k+1}} |a_{\bar y}| \cdot \length(\bar y)^n,\]
where $\length(\bar y) = \max_{0 \le i,j \le k} d(y_i,y_j)$ for $\bar y = (y_0, \ldots, y_k)$.

We equip $\Cuf_k(Y)$ with the family of norms $(\|\largecdot\|_{\infty,n} + \|\partial \largecdot\|_{\infty,n})_{n \in \IN_0}$, denote its completion to a \Frechet space by $\Cufpol_k(Y)$ and the resulting homology by $\Hufpol_\ast(Y)$.
\end{defn}

We want to find now conditions under which the wrong way map will be continuous and therefore give rise to a map $\Hufpol_\ast(\Gamma_M) \to \Hufpol_{\ast-q}(\Gamma_N)$.

\begin{defns}
We will introduce now polynomial dependences into the definitions from the last section.
\begin{enumerate}
\item A metric space $X$ is said to be polynomially $q$-connected if for every $i \le q$ we have the following: there exists a polynomial $P$ such that if $S^i \to X$ is an $L$-Lipschitz map, then there exists a contraction of it which is $P(L)$-Lipschitz.
\item $(N,h) \subset (M,g)$ is a polynomial ulf-submanifold if $\id\colon (N, d_h) \to (N, d_g|_{N\times N})$ and its inverse $\id\colon (N, d_g|_{N\times N}) \to (N, d_h)$ are uniformly continuous and polynomially rough. The latter means that in the definition of a rough map the $S$, viewed as a function of $R$, is bounded from above by a polynomial in $R$.
\item $\Gamma_M$ has polynomial growth if there exists a polynomial $P$ such that $\card\{B_r(\gamma)\} \le P(r)$ for all $\gamma \in \Gamma_M$ and all $r > 0$.
\end{enumerate}
These polynomial dependences will be needed in the proof of Theorem~\ref{thm:continuity} below.
\end{defns}

\begin{example}\label{exjnkwe23}
Let $M$ be a closed, connected, Riemannian manifold with $\pi_i(M) = 0$ for $2 \le i \le q-1$ and with $\pi_1(M)$ virtually nilpotent. Then the universal cover of $M$ will be polynomially $(q-1)$-connected and of polynomial volume growth if equipped with the pull-back metric and if we choose $\Gamma_M$ to be the translate of a single point by all deck transformations.

That the universal cover will have polynomial volume growth follows from the fact that virtually nilpotent groups have polynomial growth. That it will be polynomially $(q-1)$-connected follows from Riley's result \cite[Theorems D \& E]{riley} that virtually nilpotent groups have polynomially bounded higher-order combinatorial functions.

Let furthermore $N \subset M$ be a submanifold. Then any connected component $\bar N \subset X$ of the lift of $N$ to the universal cover $X$ of $M$ by the projection map $X \to M$ will be a polynomial ulf-submanifold if equipped with the induced Riemannian metric.
\end{example}

\begin{thm}\label{thm:continuity}
Let $M$ be a Riemannian manifold, $N \subset M$ a polynomial ulf-submanifold of codimension $q \ge 1$ with a uniformly thick normal bundle $\nu$, and assume that $\Gamma_M$ has polynomial growth.

Assume further one of the following:
\begin{enumerate}
\item That $M$ is polynomially $q$-connected and $\nu$ oriented\footnote{for the chosen coefficients $A$}.
\item That $M$ is polynomially $(q-1)$-connected, $\pi_q(N) \to \pi_q(M)$ polynomially surjective and $\nu$ is trivial.
\end{enumerate}

Then the wrong way map becomes continuous with respect to the topology given in Definition \ref{defn_pol_chains} and therefore gives rise to a map $\Hufpol_\ast(\Gamma_M) \to \Hufpol_{\ast-q}(\Gamma_N)$ such that we have the following commutative diagram:
\[\xymatrix{
\Hulf_\ast(M) \ar[r] \ar[d] & \Hulf_\ast(D\nu, S\nu) \ar[r] & \Hulf_{\ast-q}(N) \ar[d]\\
\Huf_\ast(M) \ar[rr] \ar[d] & & \Huf_{\ast-q}(N) \ar[d]\\
\Hufpol_\ast(\Gamma_M) \ar[rr] & & \Hufpol_{\ast-q}(\Gamma_N)}\]
\end{thm}

\begin{proof}
Straightforward; the assumptions are tailored such that this result holds.
\end{proof}

\subsection{Large scale higher codimensional index obstructions}
\label{seckjnsd}

The author \cite{engel_rough} constructed a map $\chi\colon \Kalg_\ast(\IC_u^\ast (M)) \to \Huf_\ast(M)$ for a manifold $M$ of bounded geometry, where $\IC_u^\ast (M)$ is the algebraic uniform Roe algebra of $M$. It has the property that if $M$ is spin and if we denote by $\ind (\slashed D_M) \in \Kalg_\ast(\IC_u^\ast (M))$ the large scale index class of the Dirac operator of $M$, then $\chi(\ind (\slashed D_M))$ is given by the coarsification of the \Poincare dual of the total $\hat A$-class of $M$.

\[\xymatrix{
& & K_\ast^\alg(\IC_u^\ast(M)) \ar[r] \ar[d]^{\chi} & K_\ast(C_u^\ast (M)) \ar[d]\\
\HbdR^{m - \ast}(M) \ar[r]^{\PD} & \Hulf_\ast(M) \ar[r]^c & \Huf_\ast(M) \ar[r] & \Hufpol_\ast(\Gamma_M)
}\]

To be more concrete, the class $\chi(\ind (\slashed D_M))$ coincides with the image of $\PD({\hat A}_M)$ under the coarsification map $c \colon \Hulf_\ast(M) \to \Huf_\ast(M)$. Here we denote by $\hat A_M \in \bigoplus_{k \ge 0} \HbdR^{4k}(M)$ the total $\hat A$-class of the spin manifold $M$ and by $\PD\colon \HbdR^{m - \ast}(M) \to \Hulf_{\ast}(M)$ the \Poincare duality map, where $\HbdR^{m-\ast}(M)$ is bounded de Rham cohomology. Consequently, we get the following corollary from Theorem~\ref{thm:constr_loc_map}:

\begin{cor}\label{corjnsd239887}
Assume that we are in the setting of Theorem~\ref{thm:constr_loc_map} and additionally that $M$ and $N$ are both spin, of bounded geometry and that the normal bundle $\nu$ is trivial.

Then the wrong way map $\Huf_\ast(M) \to \Huf_{\ast-q}(N)$ maps the class $\chi(\ind (\slashed D_M)) \in \Huf_\ast(M)$ to $\chi(\ind (\slashed D_N)) \in \Huf_{\ast-q}(N)$.
\end{cor}

Let us treat now the completed version of the above, i.e., how to deal with the $K$-theory of the uniform Roe algebra $C_u^\ast (M)$ itself, which is the completion of the algebra $\IC_u^\ast (M)$ in operator norm. The reason why we want to deal with this is that the non-vanishing of the large scale index class of the Dirac operator in $K_\ast(C_u^\ast(M))$ carries geometric information (e.g., in this case $M$ does not admit in its quasi-isometry class a metric of uniformly positive scalar curvature), whereas from the non-vanishing in $\Kalg_\ast(\IC_u^\ast (M))$ one usually can not conclude this. For further information on the index theory behind all this the reader should consult Roe's article \cite{roe_coarse_cohomology}.

Note that in the above diagram the arrow $K_\ast(C_u^\ast (M)) \to \Hufpol_\ast(\Gamma_M)$ is the continuous extension of the map $\chi$.\footnote{To be concrete, one defines a \Frechet completion $C_{\mathrm{pol}}^\ast (M)$ of $\IC_u^\ast (M)$ and $\chi$ extends continuously to it. Furthermore, the $K$-theory of $C_{\mathrm{pol}}^\ast (M)$ coincides with the one of $C_u^\ast(M)$.} So if we additionally assume polynomial dependencies as in Theorem \ref{thm:continuity}, we can apply that theorem to the conclusion of the above Corollary~\ref{corjnsd239887} and get the following:

\begin{cor}
Under the assumptions of Theorem \ref{thm:continuity}\footnote{Note that this includes that $M$ has polynomial volume growth.} and Corollary~\ref{corjnsd239887}:

If $c(\PD({\hat A}_N)) \not= 0 \in \bigoplus_{k \ge 0} \Hufpol_{\dim(N)-k}(\Gamma_N)$, then $\ind(\slashed D_M) \not= 0 \in K_\ast(C_u^\ast(M))$.
\end{cor}

\section{Equivariant setting}\label{sec:4}

In Section~\ref{seck09723} we will prove an equivariant version of Theorem~\ref{thm:constr_loc_map}. In Section~\ref{subsecn2sd2} we will apply this to universal covers of closed manifolds with the action of their deck transformation group. This will result in wrong way maps in homology of groups, i.e., we get the proof of the topological version of the main theorem as stated in the introduction.

\subsection{Equivariant rough wrong way maps}\label{seck09723}

The natural type of group actions in our setting is the following one:

\begin{defn}[ulf-actions]\label{defn:ulf-action}
An action of a group $\Gamma$ on a metric space $X$ is called ulf if
\begin{itemize}
\item the family of maps $X \to X$ given by $\{x \mapsto \gamma x\}_{\gamma \in \Gamma}$ is uniformly equicontinuous and equirough\footnote{This means that the $S$ in Definition~\ref{defnjsd1f32} of rough maps depends only on $R$ and not on $\gamma$.}, and
\item the action is uniformly proper, i.e., for every $r > 0$ exists a number $K_r > 0$ such that for all $x \in X$ we have $\card \{\gamma \in \Gamma\colon \gamma B_r(x) \cap B_r(x) \not= \emptyset\} < K_r$.
\end{itemize}
Note that the first point is automatically satisfied if $\Gamma$ acts via isometries.
\end{defn}

\begin{example}\label{ex:ulf_action}
Let $(M,g)$ be a closed Riemannian manifold. Then the action of $\pi_1(M)$ on the universal cover $X$ of $M$ is ulf if we equip $X$ with the pull-back metric.
\end{example}

\begin{rem}
The idea behind Definition \ref{defn:ulf-action} is that for any continuous map $\sigma\colon \Delta^n \to X$ the sum $\sum_{\gamma \in \Gamma} \gamma \sigma$ will be a uniformly locally finite chain.
\end{rem}

We denote by $\HulfG_\ast(X)$ the $\Gamma$-equivariant uniformly locally finite homology of $X$ and by $\HufG_\ast(X)$ the $\Gamma$-equivariant uniformly finite homology of $X$. We have again a natural map $\HulfG_\ast(X) \to \HufG_\ast(X)$ by mapping a simplex to its ordered tuple of vertices.

Note that if the action of $\Gamma$ on $X$ is free and cocompact, then the equivariant uniformly locally finite homology $\HulfG_\ast(X)$ coincides with the homology $H_\ast(X / \Gamma)$ of the quotient (we need cocompactness of the action so that we know that only finitely many $\Gamma$-orbits of simplices from a chain in $\HulfG_\ast(X)$ hit a fundamental domain).

\begin{prop}\label{prop:equiv_iso}
Let the action of $\Gamma$ on $X$ be ulf and free, and let $X$ be equicontinuously contractible.

Then the map $\HulfG_\ast(X) \to \HufG_\ast(X)$ is an isomorphism.
\end{prop}

\begin{proof}
The proof is the same as the one of Proposition \ref{prop:contr_iso} with $\Gamma$-equivariance incorporated: for $\bar x \in X^{i+1}$ we construct a simplex $\Delta(\bar x)\colon \Delta^i \to X$ as before and then we use for the points $\gamma \bar x$ the simplices $\gamma \Delta(\bar x)$. Then we iterate this procedure: we consider a point $\bar x^\prime$ which does not have a corresponding simplex built for it yet, we construct the simplex, and then we translate this simplex by the group action to get corresponding simplices for all the translates of $\bar x^\prime$. We do this until all points in $X^{i+1}$ have a corresponding simplex built for them.

Freeness of the action of $\Gamma$ is needed, because otherwise we will get a problem if we have a point $\bar x \in X^{i+1}$ with $\gamma \bar x = \bar x$ but $\gamma \Delta(\bar x) \not= \Delta(\bar x)$.
\end{proof}

\begin{thm}\label{thm:equiv_loc}
Under the same assumptions as in Theorem \ref{thm:constr_loc_map} and additionally: let the action of $\Gamma$ on $M$, $N$ and the normal bundle $\nu$ be ulf, free and such that the embeddings $N \to \nu$ and $\nu \to M$ are $\Gamma$-equivariant.\footnote{Since $N \subset M$ is assumed to be a ulf-submanifold, it doesn't matter if we use the to $N$ intrinsic metric or the one induced from $M$ in order to say that the action of $\Gamma$ on $N$ is ulf.}

Then there is a map $\HufG_\ast(M) \to \HufG_{\ast-q}(N)$ such that the following diagram commutes:
\[\xymatrix{
\HulfG_\ast(M) \ar[r] \ar[d] & \HulfG_\ast(D\nu, S\nu) \ar[r] & \HulfG_{\ast-q}(N) \ar[d]\\
\HufG_\ast(M) \ar[rr] & & \HufG_{\ast-q}(N)}\]
\end{thm}

\begin{proof}
The construction of the map is analogous to the non-equivariant construction in the proof of Theorem \ref{thm:constr_loc_map}, using the idea from the above proof of Proposition \ref{prop:equiv_iso}.

Concretely, this means that we may assume the $\Delta$-operator used in the argument to be equivariant. Furthermore, the map $\eta\colon M \to N$ that we use in the construction must also be made $\Gamma$-equivariant: then a point $x \in M$ is not necessarily mapped anymore to a point $y \in N$ which minimizes the distance from $x$ to $N$, but the difference to such a map will be uniform in the distance $d(x,N)$ and this suffices for the argument. And last, using the idea from the above proof of Proposition \ref{prop:equiv_iso}, we can deduce that $\pi_q(N) \to \pi_q(M)$ can be improved to be equivariantly equicontinuously surjective (provided we are in the situation of Theorem \ref{thm:constr_loc_map} where we assume this).
\end{proof}

\subsection{Wrong way maps in homology of groups}\label{subsecn2sd2}

Let $\Gamma$ be a discrete group. We can equip it with a proper, left-invariant metric, which ensures that the action of $\Gamma$ on itself is ulf.

In what follows we will use simplicial models for $B\Gamma$, and on $E\Gamma$ we will use the induced simplicial model. We use the path metric induced from barycentric coordinates on the simplices (but we could also use other path metrics, e.g., the one derived from using the spherical metrics on the simplices).

\begin{defn}\label{defnjksdf23}
Let $\Gamma$ be finitely generated. We define
\[\RHulfG_\ast(E\Gamma) := \colim_{K \subset B\Gamma} \HulfG_\ast(\widetilde{K}) \text{ and } \RHufG_\ast(E\Gamma) := \colim_{K \subset B\Gamma} \HufG_\ast(\widetilde{K}),\]
where the colimit runs over all finite and connected $K \subset B\Gamma$ with $\pi_1(K) \twoheadrightarrow \pi_1(B\Gamma)$, and $\widetilde{K} \subset E\Gamma$ is the preimage of $K$.
\end{defn}

\begin{lem}
The above is well-defined, i.e., the groups are independent of the choice of model for $B\Gamma$.
\end{lem}

\begin{proof}
First note that passing to a subdivision of one simplicial model for $B\Gamma$ does not change the groups.

Given two models $B\Gamma$ and $B\Gamma^\prime$, we can by the above (together with simplicial approximation) assume that there are simplicial maps $B\Gamma \to B\Gamma^\prime$ and the other way round, such that the corresponding compositions are homotopic to the corresponding identity maps.

Simplicial maps are Lipschitz with Lipschitz constant $1$, so they will intertwine the corresponding colimit systems. Moreover, the homotopies to the identities are uniformly continuous if restricted to finite sub-complexes, and so the result follows.
\end{proof}

\begin{prop}\label{prop:homology_BG}
Let $\Gamma$ be a finitely generated, discrete group.

Then we have an isomorphism $H_\ast(B\Gamma) \xrightarrow{\cong} \HufG_\ast(|\Gamma|)$, where $|\Gamma|$ denotes the group $\Gamma$ endowed with any proper, left-invariant metric.
\end{prop}

\begin{proof}
The lifting map $H_\ast(B\Gamma) \to \RHulfG_\ast(E\Gamma)$ is an isomorphism. Its inverse map is given in the following way: for $K \subset B\Gamma$ finite we choose a single simplex in a $\Gamma$-orbit $\Gamma\sigma \subset \widetilde{K}$ and push this simplex down to $K$ by the covering map $\widetilde{K} \to K$.

For any $K \subset B\Gamma$ and any equicontinuous collection of maps $\{S^i \to \widetilde{K}\}$ there exists a $K^\prime \subset B\Gamma$ such that this collection of maps is equicontinuously contractible in $\widetilde{K^\prime}$. Therefore, with a similar argument as used to prove Proposition~\ref{prop:equiv_iso}, we get an isomorphism $\RHulfG_\ast(E\Gamma) \to \RHufG_\ast(E\Gamma)$

Choosing a point in $E\Gamma$ and translating it by deck transformations we get an equivariant rough embedding of $|\Gamma|$ into $E\Gamma$. For $K \subset B\Gamma$, the lift $\widetilde{K}$ will contain this orbit for large enough $K$, but now the equivariant rough embedding $|\Gamma| \to \widetilde{K}$ is a rough equivalence (an inverse map is given by crushing a fundamental domain of $\widetilde{K}$ to a single point). So we get an isomorphism $\RHufG_\ast(E\Gamma) \to \HufG_\ast(|\Gamma|)$.

Composing everything gives the required isomorphism $H_\ast(B\Gamma) \to \HufG_\ast(|\Gamma|)$.
\end{proof}

\begin{rem}
For a compact manifold $M$ we can write the map $H_\ast(M) \to H_\ast(B\pi_1 M)$ induced by the classifying map as the composition (let us write $\Gamma := \pi_1 M$)
\begin{equation}
\label{eqrwefll222}
H_\ast(M) \stackrel{\cong}\longrightarrow \HulfG_\ast(\widetilde{M}) \longrightarrow \HufG_\ast(\widetilde{M}) \stackrel{\cong}\longleftarrow \HufG_\ast(|\Gamma|) \stackrel{\cong}\longleftarrow H_\ast(B \Gamma).
\end{equation}
The first map in the composition is given by lifting a simplex in $M$ to its orbit of simplices in the universal cover $\widetilde{M}$ of $M$, the second map is coarsification, the third map is given by choosing a single point $x_0 \in \widetilde{M}$ and translating it by $\Gamma$ (note that $\Gamma x_0$ is equivariantly coarsely equivalent to $\widetilde{M}$, if we equip $\widetilde{M}$ with the pull-back metric of any metric on $M$), and the last map in the composition is the isomorphism from Proposition~\ref{prop:homology_BG}.
\end{rem}

\begin{thm}\label{thm:loc_BG}
Let $M$ be a closed, connected manifold and let $N \hookrightarrow M$ be a connected submanifold of codimension $q \ge 1$ such that the inclusion induces an injective map on fundamental groups $\pi_1(N) \hookrightarrow \pi_1(M)$.

Assume further one of the following:
\begin{enumerate}
\item That the normal bundle $\nu$ of the embedding $N \hookrightarrow M$ is oriented\footnote{With respect to the coefficient group of the homology groups.} and, in the case $q \ge 2$, that $\pi_i(M) = 0$ for $2 \le i \le q$.
\item That the normal bundle $\nu$ of the embedding $N \hookrightarrow M$ is trivial.

Further, in the case $q = 2$ that $\pi_2(N) \to \pi_2(M)$ is surjective, and in the case $q > 2$ that $\pi_i(M) = 0$ for $2 \le i \le q-1$ and that $\pi_q(N) \to \pi_q(M)$ is surjective.
\end{enumerate}

Then we can construct a homomorphism $H_\ast(B\pi_1 M) \to H_{\ast-q}(B\pi_1 N)$ such that the following diagram commutes:
\[\xymatrix{
H_\ast(M) \ar[r] \ar[d] & H_\ast(D\nu, S\nu) \ar[r] & H_{\ast-q}(N) \ar[d]\\
H_\ast(B\pi_1 M) \ar[rr] & & H_{\ast-q}(B\pi_1 N)}\]
where the map $H_\ast(M) \to H_\ast(D\nu, S\nu)$ is the Thom--Pontryagin collapse, and the map $H_\ast(D\nu, S\nu) \to H_{\ast-q}(N)$ is cap-product with the Thom class of $\nu$. The vertical maps are induced by the classifying maps.
\end{thm}

\begin{proof}
Let us write for simplicity of notation $\pi_1(M) = \Gamma$ and $\pi_1(N) = \Lambda$. Equip $M$ with any Riemannian metric and a universal cover $X$ of $M$ with the pull-back metric. The preimage of $N$ under the covering projection $X \to M$ will be a disjoint union of copies of the universal cover of $N$, because we assume that $\Lambda \to \Gamma$ is injective. Let $\bar N \subset X$ be one of these copies and equip it with the induced Riemannian metric. Then $\bar N$ will be a ulf-submanifold and its normal bundle $\bar\nu$ will be uniformly thick. The action of the subgroup $\Lambda \subset \Gamma$ will be ulf and free on $\bar N$, $X$ and $\bar\nu$, and the embeddings $\bar N \to \bar\nu$ and $\bar\nu \to X$ will be $\Lambda$-equivariant. This gives us equivariant Thom--Pontryagin collapse maps $\HulfL_\ast(X) \to \HulfL_\ast(D\bar\nu, S\bar\nu) \to \HulfL_{\ast-q}(\bar N)$.

Let us take a look at the following diagram summarizing what we have discussed up to now. Keep in mind the maps in Equation~\eqref{eqrwefll222}, note that the dotted arrow is not yet constructed, and the horizontal maps $\HulfG_\ast(X) \to \HulfL_\ast(X)$ and $\HufG_\ast(X) \to \HufL_\ast(X)$ are given by forgetting part of the equivariance:

\[\xymatrix{
& H_\ast(M) \ar[r] \ar[dl]_\cong & H_\ast(D\nu, S\nu) \ar[d]^\cong \ar[r] & H_{\ast-q}(N) \ar[d]^\cong\\
\HulfG_\ast(X) \ar[d] \ar[r] & \HulfL_\ast(X) \ar[d] \ar[r] & \HulfL_\ast(D\bar\nu, S\bar\nu) \ar[r] & \HulfL_{\ast-q}(\bar N) \ar[d]\\
\HufG_\ast(X) \ar[r] & \HufL_\ast(X) \ar@{-->}[rr] & & \HufL_{\ast-q}(\bar N)\\
\HufG_\ast(|\Gamma|) \ar[u]^\cong & & & \HufL_{\ast-q}(|\Lambda|) \ar[u]_\cong
}\]

In view of Proposition~\ref{prop:homology_BG} it remains to show that the dotted arrow exists and makes the diagram commute. To accomplish this we will apply Theorem~\ref{thm:equiv_loc}, which means that we have to verify its assumptions. We have already discussed most of these assumptions in the first paragraph of this proof, and the equicontinuous $(q-1)$-connectedness, resp.~the $q$-connectedness, of $X$ was discussed in Example~\ref{exjnksd23}. And last, in the case where we assume that $\pi_q(N) \to \pi_q(M)$ is surjective, we conclude that the map $\pi_q(\bar N) \to \pi_q(X)$ is equicontinuously surjective. So we can apply Theorem~\ref{thm:equiv_loc}, which finishes this proof.
\end{proof}

\begin{rem}
If for some $q^\prime \ge q$ one of the assumptions in Theorem~\ref{thm:loc_BG} is satisfied for $q^\prime$ instead of $q$, an equivariant version of Remark~\ref{rem:more_maps} gives us that we will get a map $H_\ast(B\pi_1 M) \to H_{\ast-q^\prime}(B\pi_1 N)$ for any cohomology class of degree~$q^\prime$ of the pair $(D\nu, S\nu)$. In Theorem~\ref{thm:loc_BG} we use this cohomology class for the cap product $H_\ast(D\nu, S\nu) \to H_{\ast - q}(N)$ instead of the Thom class.
\end{rem}

\bibliography{./Bibliography_Wrong_way_homology_arXiv}
\bibliographystyle{amsalpha}

\end{document}